\documentclass[a4paper,11pt]{article}

\usepackage[affil-it]{authblk}

\usepackage{indentfirst}
\usepackage[margin = 1 in]{geometry}

\usepackage[numbers,sort&compress]{natbib}
\usepackage{amssymb}
\usepackage{amsmath}
\usepackage{stmaryrd}
\usepackage{empheq}
\usepackage{amsthm}
\usepackage{mathrsfs}
\usepackage{euscript}
\usepackage{soul}
\usepackage{xcolor}
\usepackage[hypertexnames=false,colorlinks,linkcolor={blue},citecolor={blue}]{hyperref}
\usepackage{graphicx}
\usepackage{float}
\usepackage{caption}
\usepackage{epsfig}
\usepackage{subcaption}
\usepackage{scalerel} 

\usepackage{enumitem}
\setlist{
leftmargin=*, 
topsep = 4pt, 
itemsep = 1pt, 
}

\usepackage[T1]{fontenc}
\usepackage{lmodern}
\usepackage{microtype}
\usepackage[english]{babel}

\usepackage{fancyhdr}
\fancypagestyle{draftfirstpage}{

\fancyhf{}
}
\numberwithin{equation}{section}

\newtheorem{theorem}{Theorem}[section]

\newtheorem{lemma}[theorem]{Lemma}
\newtheorem{proposition}[theorem]{Proposition}
\theoremstyle{definition}
\newtheorem{definition}[theorem]{Definition}

\newtheorem{question}{Question}

\newtheorem{example}[theorem]{Example}

\allowdisplaybreaks

\newcommand{\proofstep}[1]{%
  \par\medskip\noindent\textbf{#1}\enspace
}

\newcommand{\ie}{{\itshape i.e., }}
\newcommand{\eg}{{\itshape e.g., }}
\newcommand{\eqfinp}{\,.}
\newcommand{\eqfinv}{\,,}
\newcommand{\eqsepv}{\,,}
\newcommand{\eqsepquad}{\,,\quad}
\newcommand{\eqsepand}{\quad \text{and} \quad}

\usepackage{macros}

\makeatletter
\let\original@l@section\l@section
\renewcommand*{\l@section}[2]{{\let\addvspace\@gobble\original@l@section{#1}{#2}}}
\makeatother

\newcommand{\rev}[1]{\ifmmode\expandafter\mathcolor\else\expandafter\textcolor\fi{blue}{#1}}

\title{On an open question about the equivalence of monotone maximality and surjectivity condition in~$\CAT(0)$ spaces:\\ Negative and partially positive answers}

\author{Parin Chaipunya}

\affil{
{Department of Mathematics, Faculty of Science, King Mongkut's University of Technology Thonburi},

{126 Pracha Uthit Rd.},
{Bang Mod, Thung Khru, Bangkok},
{10140}, 
{Thailand.}

Email: parin.cha@kmutt.ac.th 
}

\date{}

\begin{document}
\maketitle 
\thispagestyle{draftfirstpage}

\begin{abstract}
  In a Hilbert space, it is known that a monotone operator is maximal if and only if the \emph{surjectivity condition}, \ie its resolvent has a full domain, holds.
  This equivalence is an open question in a~$\CAT(0)$ space, knowing only that the surjectivity condition implies the monotone maximality.
  In fact, there have been three approaches to generalize the monotonicity from a Hilbert space into a~$\CAT(0)$ space~$X$ --- \textbf{A:} the ones of dual space~$X^{\ast}$, \textbf{B:} of linear dual space~$X^{\diamond}$, and \textbf{C:} of tangent bundle~$TX$.
  In this paper, we answer \emph{negatively} on the first two approaches by constructing counterexamples on an~$\R$-tree.
  We then go on proving that the equivalence is \emph{affirmative} on the same~$\R$-tree, without any assumption on the domain.
  Moreover, we show that the answer is still \emph{positive} in some other cases, through a flat slice principle that transfers Minty's principle from Hilbert spaces to flat convex subsets, and through Minty's principle on proper Hilbert-Hadamard spaces satisfying a convex half-space condition.
  The question of equivalence \emph{remains open} under the general setting with the tangential approach.
	
	
\end{abstract}

\normalsize

{
  \footnotesize
  \tableofcontents
}

\section{Introduction}

There have been three main approaches in the literature to generalize the concept of monotone operators from a Hilbert space into a complete~$\CAT(0)$ space~$X$.
Each of these approaches considers a set-valued map~$A : X \multimap Y$ with different choices of the codomain~$Y$.
Let us describe each of the settings and the corresponding \emph{resolvent operator}~$J_{A}^{\lambda} : X \multimap X$, where~$\lambda > 0$ is arbitrary, as follows.
\paragraph{Approach A.}
This approach is due to the study of \citet{KR} where one sets~$Y = X^{\ast}$ as the dual space in the sense of \citet{AA}.
The operator~$A : X \multimap X^{\ast}$ is said to be \emph{monotone} if~
\begin{align} \label{eq: monotone dual}
  \langle x^{\ast}, \vec{yx} \rangle - \langle y^{\ast}, \vec{yx} \rangle \geq 0
  \eqsepquad
  \bigl( \forall (x, x^{\ast}), (y , y^{\ast}) \in \gr A\bigr)
  \eqfinp
\end{align}
The resolvent of~$A$ is given by~
\begin{align} \label{eq: resolvent dual}
  J_{A}^{\lambda}(x) := \set{ z \in X \mid \lambda^{-1}\vec{zx} \in Az } \eqfinp
\end{align}

\paragraph{Approach B.}
The second approch, due to \citet{CK}, is similar to the first one except that the linear dual space~$Y = X^{\diamond}$ is used instead of~$X^{\ast}$. 
The operator~$A$ is said to be \emph{monotone} if it satisfies~\eqref{eq: monotone dual}.
Note that~$\gr A$ is now a subset of~$X \times X^{\diamond}$ instead of~$X \times X^{\ast}$.
Its corresponding resolvent~$J_{A}^{\lambda}$ is also defined as~\eqref{eq: resolvent dual}.

\paragraph{Approach C.}
The last approach that uses the tangent bundle~$Y = TX$ was studied by \citet{CKK}.
We say that~$A$ is a \emph{monotone vector field} if~$Ax \subset T_{x}X$ for all~$x \in X$ and
\begin{align} \label{eq: mono tangential}
  \langle x^{\ast}, \log_{x} y \rangle + \langle y^{\ast}, \log_{y} x \rangle \leq 0
  \eqsepquad
  \bigl( \forall (x, x^{\ast}), (y , y^{\ast}) \in \gr A\bigr)
  \eqfinp
\end{align}
The resolvent of~$A$ is then defined by~$J_{A}^{\lambda}(x) := \set{ z \in X \mid \lambda^{-1} \log_{z} x \in Az }$.

According to~\cite{CKK}, the first two approaches are developed with the Hilbert space model in mind while in general a~$\CAT(0)$ space can be vastly different when there are branches and stricy curvatures.
In particular, the Approaches A and B both \emph{distinguish} the endpoints of two different geodesics that cannot be distinguished by the Alexandrov angle (and hence the Riemannian angle too).
The Approach C was therefore introduced to produce the first-order structure that depends more on the local structures and thus more coherent to how things are treated on Riemannian manifolds.

In each of the above approaches,~$A$ is said to be \emph{maximally monotone} if it is monotone and its graph is not properly contained in the graph of any other~$B$ that is monotone in the same approach.

When~$X = \mathcal{H}$ is a Hilbert space, then all of the three approaches could be identified with a self map~$A : \mathcal{H} \multimap \mathcal{H}$ and the monotonicity defined in each of the approaches reduces to the classical monotonicity
\begin{align*}
  \langle x^{\ast} - y^{\ast}, x - y \rangle \geq 0
  \eqsepquad
  \bigl( \forall (x, x^{\ast}), (y , y^{\ast}) \in \gr A\bigr)
  \eqfinp
\end{align*}
In this setting, the resolvent is~$J_{A}^{\lambda}(x) = \set{ z \in \mathcal{H} \mid x \in (I + \lambda A)(z) } = (I + \lambda A)^{-1}(x)$.
The resolvent is at most single-valued when~$A$ is monotone.
It has nonempty values everywhere if and only if~$(I + \lambda A)$ is surjective.
If~$A : \mathcal{H} \multimap \mathcal{H}$ is monotone,~$(I + \lambda A)$ is surjective for every~$\lambda > 0$ if and only if~$A$ is maximally monotone.
This result is known in the literature as the \emph{Minty's principle}~\cite{Minty}.

In all of the approaches, we say that~$A$ satisfies the \emph{surjectivity condition} if its resolvent~$J_{A}^{\lambda}$ has nonempty values for every~$\lambda > 0$ on the whole space~$X$.
This paper concerns with the following question.

\begin{question} \label{question}
  Suppose that~$X$ is a complete~$\CAT(0)$ space and~$A$ is monotone in either of the Approaches A--C.
  Is the surjectivity condition of~$A$ equivalent to its monotone maximality?
\end{question}

In fact, this question was posed by~\citet[Section~4]{KR},~\citet{CK} and~\citet{CKK}.
It was recently restated in~\cite[Question~4.2]{SL}, together with the question of whether the assumption~$\dom A = X$ can be omitted.
It is known that the surjectivity condition implies the maximality of~$A$ when it is monotone.
Moreover, it is known under all approches \cite{KR, CK, CKK} that the subdifferential, defined by the underlying approach, of a proper convex lower semicontinuous function is maximally monotone and enjoys the surjectivity condition.
In the general setting, the question remains open eversince.

The aim of this paper is to provide some answers to this question.
As it turns out, the answers are negative for the Approaches A and B.
On the contrary, the situation is very different in the Approach C.
We are able to prove in some special settings that the equivalence hold in this latter approach.
More precisely, our contributions are threefold.
First, for the Approaches A and B, we exhibit on an~$\R$-tree a monotone operator with full domain such that no monotone extension of it has a nonempty resolvent at the branch point.
Second, for the Approach C, we prove a flat slice principle which, under a compatibility condition, reduces the existence of resolvents on a flat convex subset to Minty's principle in a Hilbert space.
Applied to the same~$\R$-tree, the flat slice principle shows that every maximally monotone vector field satisfies the surjectivity condition, without any assumption on its domain, which also answers there the question raised in~\cite{SL} of whether the full domain assumption can be dropped.
Third, we prove Minty's principle for maximally monotone vector fields with full domain on proper Hilbert-Hadamard spaces satisfying a convex half-space condition, which covers closed convex subsets of hyperbolic spaces with nonempty boundary.
Ultimately, the Question~\ref{question} remains open in the general case under Approach C.

The paper is organized as follows.
Section~\ref{sec: preliminaries} recalls the basics of~$\CAT(0)$ spaces including its dual spaces and the tangent cones.
Section~\ref{sec: approaches A and B} treats the Approaches A and B.
Section~\ref{sec: approach C} treats the Approach C: the flat slice principle is proved in Section~\ref{subsec: flat slices} and applied to~$\Omega$ in Section~\ref{subsec: revisit Omega}, and Hilbert-Hadamard spaces are studied in Section~\ref{subsec: HH}.

\section{Preliminaries} \label{sec: preliminaries}
In this section, we recall some fundamental notions of~$\CAT(0)$ spaces and the development of its first-order structure.
One may resort to \citet{BH} and \cite{BBI} for general introduction to geodesic and~$\CAT(0)$ spaces, and to \citet{Bacak} for the development of convex optimization in~$\CAT(0)$ spaces.

We recall some fundamental concepts of~$\CAT(0)$ spaces and convex analysis in this space.
Consider a metric space~$(X,d)$.
A curve~$\gamma : \ccint{0,1} \to X$ is called a \emph{geodesic} joining~$x := \gamma(0)$ to~$y := \gamma(1)$ if~$d(\gamma(t), \gamma(t')) = \abs{t - t'}d(x, y)$ for all~$t, t' \in \ccint{0, 1}$.
It is a \emph{zero geodesic} if~$\gamma(0) = \gamma(1)$, otherwise it is said to be \emph{nonzero}.
If any two~$x, y \in X$ can be joined by a (unique) geodesic, we say that~$X$ is (uniquely) geodesic.
A uniquely geodesic metric space~$(X,d)$ is a \emph{$\CAT(0)$ space} if
\begin{align} \label{eq: CN ineq}
  d(\gamma(t),z)^{2} \leq (1-t)d(\gamma(0),z)^{2} + td(\gamma(1),z)^{2} - t(1-t)d(\gamma(0), \gamma(1))^{2}
\end{align}
for any~$z \in X$, any geodesic~$\gamma$, and any~$t \in \ccint{0,1}$.

\paragraph{Dual space.}
For any quadruple of points~$x,y,u,v$ in a metric space~$(X,d)$, the corresponding \emph{quasilinearization}, originally introduced by \citet{BN}, is defined by
\begin{align*}
  \langle \vec{xy}, \vec{uv} \rangle := \frac{1}{2} \bigl( d(x, v)^{2} + d(y,u)^{2} - d(x,u)^{2} - d(y,v)^{2}  \bigr) 
  \eqfinp
\end{align*}
The \emph{dual space}~$X^{\ast}$ of~$X$ was introduced by \citet{AA}.
It consists of equivalence classes~$[t \vec{xy}]$ under the equivalence relation~$\sim$, where~$t \in \R$ and~$x,y \in X$, and~$(t,x,y) \sim (s,u,v)$ if~$t\langle \vec{xy} , \vec{wz} = s \langle \vec{uv} , \vec{wz} \rangle$ for all~$w,z \in X$.
For simplicity, we write~$t\vec{xy}$ instead of~$[t\vec{xy}]$.

\paragraph{Linear dual space.}
By definition of~$\sim$, each element~$x^{\ast} = t\vec{xy} \in X^{\ast}$ is determined by the real-valued function
\begin{align*}
  (w,z) \mapsto \langle x^{\ast}, \vec{wz} \rangle := t \langle \vec{xy}, \vec{wz} \rangle
\end{align*}
on~$X \times X$, and two elements of~$X^{\ast}$ coincide if and only if the corresponding functions coincide.
We may thus regard~$X^{\ast}$ as a subset of the vector space~$\R^{X \times X}$ of all real-valued functions on~$X \times X$.
This subset is closed under multiplication by real numbers, since~$s(t\vec{xy}) = (st)\vec{xy}$, but in general it is not closed under addition.
Following~\citet{CK}, the \emph{linear dual space} of~$X$ is the linear span
\begin{align*}
  X^{\diamond} := \operatorname{span} X^{\ast} = \set[Big]{ \sum_{k=1}^{m} x_{k}^{\ast} \mid m \in \mathbb{N},\ x_{1}^{\ast}, \dots, x_{m}^{\ast} \in X^{\ast} } \subset \R^{X \times X}
  \eqfinp
\end{align*}
The pairing with~$X^{\ast}$ extends linearly to~$X^{\diamond}$, namely~$\bigl\langle \sum_{k} x_{k}^{\ast}, \vec{wz} \bigr\rangle := \sum_{k} \langle x_{k}^{\ast}, \vec{wz} \rangle$, which is well defined because elements of~$X^{\diamond}$ are functions on~$X \times X$.
In particular,~$X^{\ast} \subset X^{\diamond}$

\paragraph{Tangent cones.}
We now recall the first-order structure of a~$\CAT(0)$ space~$X$ given by its tangent cones, following the sources~\cite{BH, BBI, CKK}.
For~$x, y, z \in X$ with~$y \neq x \neq z$, let~$\bar{\angle}_{x}(y,z) \in \ccint{0,\pi}$ be the \emph{comparison angle}, \ie the angle at the vertex corresponding to~$x$ of a Euclidean triangle with side lengths~$d(x,y)$,~$d(x,z)$ and~$d(y,z)$.
Let~$\gamma$ and~$\eta$ be the nonzero geodesics joining~$x$ to~$y$ and~$x$ to~$z$, respectively.
Since~$X$ is a~$\CAT(0)$ space, the comparison angle~$\bar{\angle}_{x}(\gamma(t), \eta(t'))$ is nondecreasing in~$t, t' \in \ooint{0,1}$ \cite[Prop.~II.3.1]{BH}.
The \emph{Alexandrov angle} between~$y$ and~$z$ at~$x$ is therefore well defined by
\begin{align*}
  \angle_{x}(y,z) := \lim_{t, t' \downarrow 0} \bar{\angle}_{x}\bigl(\gamma(t), \eta(t')\bigr)
  \eqsepand
  \angle_{x}(y,z) \leq \bar{\angle}_{x}(y,z)
  \eqfinp
\end{align*}

The Alexandrov angle is a pseudometric on the set of nonzero geodesics issuing from~$x$.
Identifying two such geodesics whenever their angle is zero, we obtain a metric space whose completion~$S_{x}X$ is called the \emph{space of directions} at~$x$.
The direction of the geodesic joining~$x$ to~$y \neq x$ is denoted by~$\xi_{x}(y) \in S_{x}X$.
The \emph{tangent cone}~$T_{x}X$ at~$x$ is the Euclidean cone over~$S_{x}X$, that is, the quotient of~$\coint{0,\infty} \times S_{x}X$ obtained by collapsing~$\set{0} \times S_{x}X$ to a single point~$0$, called the \emph{apex}, endowed with the metric
\begin{align*}
  d_{T}\bigl( (s, \xi), (t, \eta) \bigr)^{2} := s^{2} + t^{2} - 2st \cos \angle(\xi, \eta)
  \eqfinp
\end{align*}
We write~$s\xi$ instead of the class of~$(s,\xi)$, and~$\norm{u} := d_{T}(0, u)$ for~$u \in T_{x}X$.
The multiplication~$\lambda(s\xi) := (\lambda s)\xi$ by~$\lambda \geq 0$ is well defined on~$T_{x}X$, and so is the \emph{product}
\begin{align*}
  \langle u, v \rangle := \frac{1}{2} \bigl( \norm{u}^{2} + \norm{v}^{2} - d_{T}(u,v)^{2} \bigr)
  \eqsepv
\end{align*}
which further yields~$\langle s\xi, t\eta \rangle = st \cos \angle(\xi, \eta)$.
It is positively homogeneous in each variable, but~$T_{x}X$ is in general not a linear space.
The \emph{tangent bundle} of~$X$ is the disjoint union~$TX := \bigsqcup_{x \in X} T_{x}X$.

For~$x, y \in X$, the \emph{logarithm map} (or the inverse of the \emph{exponential map} in the language of manifolds) is defined by~$\log_{x} y := d(x,y) \, \xi_{x}(y)$ if~$y \neq x$, and~$\log_{x} x := 0$.
Thus~$\norm{\log_{x} y} = d(x,y)$, and~$\log_{x} y$ points along the geodesic joining~$x$ to~$y$.
Moreover,~$\langle \log_{x} y, \log_{x} z \rangle = d(x,y) d(x,z) \cos \angle_{x}(y,z)$, and the inequality~$\angle_{x}(y,z) \leq \bar{\angle}_{x}(y,z)$ yields the comparison inequality
\begin{align} \label{eq: comparison tangent}
  \langle \log_{x} y, \log_{x} z \rangle \geq \langle \vec{xy}, \vec{xz} \rangle = \frac{1}{2} \bigl( d(x,y)^{2} + d(x,z)^{2} - d(y,z)^{2} \bigr)
  \eqfinp
\end{align}
When~$X = \mathcal{H}$ is a Hilbert space, the tangent cone~$T_{x}\mathcal{H}$ is identified with~$\mathcal{H}$ and~$\log_{x} y$ with~$y - x$.
When~$X$ is a Hadamard manifold, the tangent cone~$T_{x}X$ is identified with the usual tangent space, and~$\log_{x}$ with the inverse of the exponential map at~$x$ \cite{Sakai}.

We will also need the following continuity properties.
The Alexandrov angle~$\angle_{x}(y,z)$ is upper semicontinuous as a function of~$(x,y,z)$ on the set where~$x \notin \set{y,z}$, and, for a fixed~$x$, the function~$(y,z) \mapsto \angle_{x}(y,z)$ is continuous on~$(X \setminus \set{x})^{2}$ \cite[Prop.~II.3.3]{BH}.
Consequently, for fixed~$y,z \in X$, the function~$x \mapsto \langle \log_{x} y, \log_{x} z \rangle$ is lower semicontinuous on~$X$.
Indeed, at~$x \notin \set{y,z}$, it equals~$d(x,y)d(x,z)\cos\angle_{x}(y,z)$, where the distances are continuous and positive, and~$\cos\angle_{x}(y,z)$ is lower semicontinuous because the cosine is decreasing on~$\ccint{0,\pi}$.
At~$x \in \set{y,z}$, it vanishes, while its absolute value is at most~$d(x,y)d(x,z)$, which tends to~$0$ as~$x$ approaches~$y$ or~$z$.

Moreover, for fixed~$x \in X$ and~$v \in T_{x}X$, the function~$y \mapsto \langle v, \log_{x} y \rangle$ is continuous on~$X$.
To see this, first note that~$\log_{x} : X \to T_{x}X$ is continuous, since the inequality~$\angle_{x}(y,z) \leq \bar{\angle}_{x}(y,z)$ gives
\begin{align*}
  d_{T}(\log_{x} y, \log_{x} z)^{2} = d(x,y)^{2} + d(x,z)^{2} - 2d(x,y)d(x,z)\cos\angle_{x}(y,z) \leq d(y,z)^{2}
  \eqfinp
\end{align*}
Then the claim follows from the identity~$\langle v, \log_{x} y \rangle = \frac{1}{2}\bigl( \norm{v}^{2} + d(x,y)^{2} - d_{T}(v, \log_{x} y)^{2} \bigr)$, whose RHS is continuous in~$y$.

\section{Approaches A and B: Negative answers} \label{sec: approaches A and B}

Throughout the section, suppose that~$\Omega$ is constructed by gluing three copies of the ray~$\R_{+}$ at its zero~$0$.
Let us denote by~$O$ its branch point and for~$i = 1,2,3$, let~$e_{i}$ the point of unit distance from~$O$ on the ray~$i$.
Then we may uniquely express each point~$x_{i} \neq O$ in the ray~${i}$ by~$re_{i}$, where~$r > 0$ is the length away from~$O$.
We write~$e_{i} := 1e_{i}$.
The distance on~$\Omega$ is then define by~$\rho(O,re_{i}) = r$,~$\rho(re_{i}, r'e_{i}) = \abs{r - r'}$, and~$\rho(re_{i}, r'e_{j}) = r + r'$ for~$r, r' > 0$,~$i,j = 1,2,3$ with~$i \neq j$.
Then~$(\Omega, \rho)$ is a complete~$\CAT(0)$ space.

Let us fix an element~$x^{\ast} := \vec{oe_{1}} \in \Omega^{\ast}$.
The following calculations are obvious and we state them without a proof.
\begin{lemma} \label{lem: aux}
  Put~$h(x) := \frac{1}{2} \bigl( \rho(O,x)^{2} - \rho(e_{1},x)^{2} \bigr)$ for~$x \in \Omega$.
  Then
  \begin{enumerate}[label=\upshape(\alph*), leftmargin=*]
    \item \label{conc a: lem aux}
      $\langle x^{\ast} , uv \rangle = h(v) - h(u)$ for all~$u,v \in \Omega$.
    \item \label{conc b: lem aux}
      For any~$r > 0$,~$h(re_{1}) = r - \frac{1}{2}$ and~$h(re_{i}) = -r - \frac{1}{2}$ for~$i = 2,3$.
    \item \label{conc c: lem aux}
      For any~$z,u \in \Omega$ and~$\lambda > 0$,~$\langle \lambda^{-1}\vec{zO} , \vec{uz} \rangle = \frac{1}{2\lambda} \bigl( \rho(O,u)^{2} - \rho(z,u)^{2} - \rho(O,z)^{2} \bigr)$.
  \end{enumerate}
\end{lemma}

\begin{lemma} \label{lem: anti mono}
  Let~$\lambda > 0$.
  For any~$z \in \Omega$, there exists~$u \in \Omega$ such that 
  \begin{align} \label{eq: anti mono}
    \langle \lambda^{-1} \vec{zO} , \vec{uz} \rangle - \langle x^{\ast} , \vec{uz} \rangle < 0
    \eqfinp
  \end{align}
\end{lemma}
\begin{proof}
  For any~$z, u \in \Omega$, by~\ref{conc a: lem aux} and~\ref{conc c: lem aux} of Lemma~\ref{lem: aux}, the LHS of~\eqref{eq: anti mono} equals
  \begin{align*}
    \Delta(z,u) := \frac{1}{2\lambda} \bigl( \rho(O,u)^{2} - \rho(z,u)^{2} - \rho(O,z)^{2} \bigr) - \bigl( h(z) - h(u) \bigr)
    \eqfinp
  \end{align*}
  We consider the different cases of~$z$.
  \proofstep{Case I: $z = O$.} Choose~$u = e_{2}$ and one obtains~$\Delta(z,u) = -1$.
  \proofstep{Case II: $z = re_{1}$,~$r > 0$.} Take~$u = re_{2}$ and one gets~$\Delta(z,u) = -\frac{2r^{2}}{\lambda} - 2r < 0$.
  \proofstep{Case III: $z = re_{2}$,~$r > 0$.} Take~$u = re_{2}$ and we obtain~$\Delta(z,u) = -\frac{2r^{2}}{\lambda} - 2r < 0$.
  \proofstep{Case IV: $z = re_{3}$,~$r > 0$.} Take~$u = re_{2}$ and hence~$\Delta(z,u) = -\frac{2r^{2}}{\lambda} < 0$. \\
  It thus follows that for any~$z \in \Omega$, there exists~$u \in \Omega$ in which~\eqref{eq: anti mono} holds.
\end{proof}

The following theorem gives a negative answer to Question~\ref{question}.
\begin{theorem} \label{thm: main dual}
  Define~$A : \Omega \multimap \Omega^{\ast}$ by~$Ax := \set{x^{\ast}}$ for all~$x \in \Omega$.
  Then
  \begin{enumerate}[label=\upshape(\alph*), leftmargin=*]
    \item \label{conc a: thm main dual}
      $A$ is monotone and~$\dom A = \Omega$.
    \item \label{conc b: thm main dual}
      If~$B : \Omega \multimap \Omega^{\ast}$ is a monotone extension of~$A$, then~$J_{B}^{\lambda}(O) = \emptyset$ for all~$\lambda > 0$.
  \end{enumerate}
\end{theorem}
\begin{proof}
  \ref{conc a: thm main dual}
  Since~$A$ is a constant map, the LHS of~\eqref{eq: monotone dual} is always~$0$ and hence~$A$ is monotone.
  The fact that~$\dom A = \Omega$ is obvious.

  \ref{conc b: thm main dual}
  Let~$B$ be a monotone extension of~$A$ and suppose to the contrary that~$J_{B}^{\lambda}(O) \neq \emptyset$ for some~$\lambda > 0$.
  Suppose that~$z \in J_{B}^{\lambda}(O)$, which means~$\lambda^{-1}\vec{zO} \in Bz$.
  Let~$u \in \Omega$ be an element corresponding to~$z$ as in Lemma~\ref{lem: anti mono}.
  Since~$\set{x^{\ast}} = Au \subset Bu$, the inequality~\eqref{eq: anti mono} contradicts with the fact that~$B$ is monotone.
  Therefore~$J_{B}^{\lambda}(O) = \emptyset$ for all~$\lambda > 0$.
\end{proof}

\paragraph{Negative answers.}
Since every monotone operator~$A : \Omega \multimap \Omega^{\ast}$ admits a maximally monotone extension~$\hat{A}$, the above theorem yields~$J_{\hat{A}}^{\lambda}(O) = \emptyset$ for all~$\lambda > 0$.
This answers Question~\ref{question} \emph{negatively} for the Approach A.
Since~$\Omega^{\ast} \subset \Omega^{\diamond}$, the same proof and reasoning gives also a \emph{negative} answer for the Approach B.

\section{Approach C: Partial affirmative answers} \label{sec: approach C}
Take the complete~$\CAT(0)$ space~$(\Omega, \rho)$ as in the Section~\ref{sec: approaches A and B}.
The aim of this section is to show that the counterpart construction of Section~\ref{sec: approaches A and B} does not give a negative answer under the tangential Approach C.
Interestingly, we obtain a contrasting conclusion that any maximally monotone vector field~$A : \Omega \multimap T\Omega$ over the same space~$(\Omega, \rho)$ has the property that~$\dom J_{A}^{\lambda} = \Omega$ for all~$\lambda > 0$.
This is indeed an \emph{affirmative} answer for the space~$\Omega$.

\paragraph{Partial affirmative answers.}
This section gives \emph{partial affirmative} answers to Question~\ref{question} under the Approach C in different special cases.
In Section~\ref{subsec: flat slices}, we establish a flat slice principle (Theorem~\ref{thm: flat slice principle}), which transfers Minty's principle from a Hilbert space to a flat slice~$C \subset X$, provided that the restriction of~$A$ to~$C$ is compatible with~$A$.
In Section~\ref{subsec: revisit Omega}, we apply this principle to~$(\Omega, \rho)$ and show that every maximally monotone vector field on~$\Omega$ satisfies the surjectivity condition (Theorem~\ref{thm: maxim mono surj tripod}), without any assumption on its domain.
In Section~\ref{subsec: HH}, we turn to Hilbert-Hadamard spaces, whose tangent cones are flat, and prove the Minty's principle (Theorem~\ref{thm: Minty HH}) for maximally monotone vector fields with full domain on a proper Hilbert-Hadamard space satisfying a convex half-space condition.
Since~$\Omega$ is not a Hilbert-Hadamard space, the two results are complementary.

\subsection{Flat slices} \label{subsec: flat slices}
In fact, the affirmative answer is based on a more general situation where one could make a slice that is isometric to a closed convex set of a Hilbert space.
Let us give a formal definition to the idea.
\begin{definition} \label{def: flat slice}
  A closed convex subset~$C$ of a complete~$\CAT(0)$ space is called a \emph{flat slice} if there is an isometry~$\varphi$ from~$C$ onto a closed convex subset~$K$ of a Hilbert space~$\mathcal{H}$ with nonempty relative interior~$\relinterior K \neq \emptyset$.
\end{definition}
Suppose that~$C \subset X$ is a flat slice as in the above definition.
For any~$x,y,y' \in C$, the segments~$\ccint{x,y}$ and~$\ccint{x,y'}$ in~$C$ are mapped by~$\varphi$ to the segments~$\ccint{\varphi(x), \varphi(y)}$ and~$\ccint{\varphi(x), \varphi(y')}$ in~$K$, respectively.
The Alexandrov angle between them therefore reduces to the Euclidean angle between the vectors~$\varphi(y) - \varphi(x)$ and~$\varphi(y') - \varphi(x)$.
One consequently has
\begin{align*}
  \langle \log_{x} y , \log_{x} y' \rangle = \langle \varphi(y) - \varphi(x) , \varphi(y') - \varphi(x) \rangle
  \eqfinp
\end{align*}
By this identity, the map~$t \log_{x} y \mapsto t\bigl( \varphi(y) - \varphi(x) \bigr)$ is well-defined and preserves the product.
It therefore extends to an isometric of~$T_{x}C \subset T_{x}X$ onto the tangent cone~$T_{K}(\varphi(x)) := \cl \left[ \bigcup_{t \geq 0} t \bigl(K - \varphi(x)\bigr) \right]$, which is a closed convex cone in~$\mathcal{H}$.

It thus makes sense to identify~$C$ with~$K$ through~$\varphi$ and~$T_{x}C$ with~$T_{K}(\varphi(x))$ through the isometry described above.
In particular,~$\log_{x} y$ is identified with~$y - x$ and the product on~$T_{x}C$ is the product on~$\mathcal{H}$. 
For a vector field~$A : X \multimap TX$, put~$A_{C}(x) := Ax \cap T_{x} C$ for each~$x \in C$.
Through the above identification,~$A_{C}$ is viewed as a set-valued map from~$C \subset \mathcal{H}$ to~$\mathcal{H}$.

\begin{theorem} \label{thm: flat slice principle}
  Let~$A : X \multimap TX$ be a maximally monotone vector field and~$C \subset X$ a flat slice such that
  \begin{enumerate}[label=\upshape(\roman*), leftmargin=*]
    \item \label{cond a: thm flat slice principle}
      either $\dom A_{C} \cap \relinterior C \neq \emptyset$, where~$\dom A_{C} := \set{x \in C \mid A_{C}(x) \neq \emptyset}$, or~$C$ is isometric to a Hilbert space;
    \item \label{cond b: thm flat slice principle}
      for every~$x \in C$ and~$w \in T_{x}C$,
      \begin{align*}
        & \langle w, \log_{x} y \rangle + \langle v, \log_{y} x \rangle \leq 0 \quad \forall (y,v) \in \gr A_{C} \\
        & \quad \implies
        \langle w, \log_{x} y \rangle + \langle v, \log_{y} x \rangle \leq 0 \quad \forall (y,v) \in \gr A 
        \eqfinp
      \end{align*}
  \end{enumerate}
  Then~$J_{A}^{\lambda}(x) \neq \emptyset$ for all~$x \in C$ and~$\lambda > 0$.
\end{theorem}
\begin{proof}
  Let~$C$ be isometric to~$K$ as in the Definition~\ref{def: flat slice}.
  Without loss of generality, one may assume that the affine hull of~$K$ is a Hilbert space~$\mathcal{H}$ and so~$\relinterior K = \interior K$.

  Through the identification,~$A_{C}$ is a monotone operator from~$C$ into~$\mathcal{H}$.
  Extend it to a maximally monotone operator~$\tilde{A} : \mathcal{H} \multimap \mathcal{H}$.
  Consider~$N_{C}$ the normal cone of~$C$ which is also maximally monotone and with~$\dom N_{C} = C$.
  We show that~$B := \tilde{A} + N_{C}$ is maximally monotone on~$\mathcal{H}$.
  If~$C$ is isometric to a Hilbert space, then~$K$ is a closed affine subspace of~$\mathcal{H}$, because isometries between Hilbert spaces are affine.
  Hence~$C = K = \mathcal{H}$ under the above normalization, so that~$N_{C}(z) = \set{0}$ for all~$z \in \mathcal{H}$ and~$B = \tilde{A}$ is maximally monotone.
  Otherwise, since~$\gr A_{C} \subset \gr \tilde{A}$, we have~$\dom A_{C} \subset \dom \tilde{A}$, and hence~$\dom \tilde{A} \cap \relinterior C \neq \emptyset$ by~\ref{cond a: thm flat slice principle}.
  Since~$\relinterior C = \interior C = \interior \dom N_{C}$, we have~$\dom \tilde{A} \cap \interior \dom N_{C} \neq \emptyset$.
  Using Rockafellar's sum theorem \cite{R}, the operator~$B$ is then maximally monotone on~$\mathcal{H}$.
  Let~$x \in C$ and~$\lambda > 0$.
  From the Minty's principle~\cite{Minty} (over~$\mathcal{H}$),~$I + \lambda B$ is surjective.
  Hence, there exists~$z \in \mathcal{H}$ such that~$x \in z + \lambda B(z)$.
  Rearranging this inclusion gives
  \begin{align*}
    w := \lambda^{-1}(x - z) \in B(z) = \tilde{A}(z) + N_{C}(z)
    \eqfinp
  \end{align*}
  This implies that~$N_{C}(z) \neq \emptyset$, we have~$z \in C$.
  We may write~$w = \tilde{a} + n$ for some~$\tilde{a} \in \tilde{A}(z)$ and~$n \in N_{C}(z)$.
  Because~$x - z \in C - z \subset T_{z} C$ and~$T_{z} C$ is a cone, we have~$w \in T_{z} C$.

  We claim that~$\langle w - v , z - y \rangle \geq 0$ for any~$(y,v) \in \gr A_{C}$.
  Since~$(y,v) \in \gr A_{C} \subset \gr \tilde{A}$ and~$(z, \tilde{a}) \in \gr \tilde{A}$ and~$\tilde{A}$ is monotone, one obtains~$\langle \tilde{a} - v , z - y \rangle \geq 0$.
  Also, the fact that~$y \in C$ yields~$\langle n, y - z \rangle \leq 0$.
  Combining the two facts, one proves the claim by
  \begin{align*}
    \langle w - v, z - y \rangle = \langle \tilde{a} - v , z - y \rangle + \langle n , z - y \rangle \geq 0
    \eqfinp
  \end{align*}
  In view of the hypothesis~\ref{cond b: thm flat slice principle}, the claim is true for all~$(y, v) \in \gr A$.
  Since~$A$ is maximally monotone, we have~$w \in A(z)$.
  Through the identification of~$w$, we get that~$\lambda^{-1} \log_{z}x \in A(z)$.
  Equivalently, this means~$z \in J_{A}^{\lambda}(x)$ and the proof is thus complete.
\end{proof}

\subsection{A revisit to~$(\Omega, \rho)$} \label{subsec: revisit Omega}
Before we proceed, we note that the tangent cone~$T_{x}\Omega$ at~$x \neq \Omega \setminus \set{O}$ can be identified with the real line~$\R$.
This allows us to identify~$A(re_{i})$ with~$a_{i}(r) \subset \R$ when~$r > 0$.
The tangent cone~$T_{O}\Omega$ consists of vectors~$sd_{i}$ with~$s \geq 0$ and~$d_{i}$ is the unit direction of ray~$i$.
We also have~$\langle sd_{i}, td_{j} \rangle = st$ if~$i = j$, and~$\langle sd_{i}, td_{j} \rangle = -st$ if~$i \neq j$.

\begin{lemma} \label{lem: equiv mono tripod}
  The vector field~$A : \Omega \multimap T\Omega$ is monotone if and only if all of the following conditions hold for all~$i,j \in \set{1,2,3}$ and~$r, r' > 0$:
  \begin{enumerate}[label=\upshape(M-\alph*), leftmargin=*]
    \item \label{cond a: lem equiv mono tripod}
      $(a - a')(r - r') \geq 0$ for all~$a \in a_{i}(r)$ and~$a' \in a_{i}(r')$;
    \item \label{cond b: lem equiv mono tripod}
      $a + a' \geq 0$ for all~$a \in a_{i}(r)$ and~$a' \in a_{j}(r')$ whenever~$i \neq j$;
    \item \label{cond c: lem equiv mono tripod}
      if~$sd_{i} \in AO$ with~$s \geq 0$, then~$a \geq s$ for all~$a \in a_{i}(r)$ and~$a \geq -s$ for all~$a \in a_{j}(r)$ with~$i \neq j$.
  \end{enumerate}
\end{lemma}
\begin{proof}
  Monotonicity requires~\eqref{eq: mono tangential} for all~$(x, x^{\ast}), (y, y^{\ast}) \in \gr A$.
  It holds trivially when~$x = y$, and it is symmetric in the two pairs.
  So, up to swapping~$x$ and~$y$, there are exactly three cases with~$x \neq y$, and we show that~\eqref{eq: mono tangential} in Cases~I, II and~III below is equivalent to~\ref{cond a: lem equiv mono tripod}, \ref{cond b: lem equiv mono tripod} and~\ref{cond c: lem equiv mono tripod}, respectively.
  Throughout,~$x = re_{i}$ with~$r > 0$, and~$x^{\ast} \in Ax$ is identified with~$a \in a_{i}(r)$.

  \proofstep{Case I: $y = r'e_{i}$ with~$r' > 0$.}
  Let~$y^{\ast} \in Ay$ be identified with~$a' \in a_{i}(r')$.
  Then~$\log_{x} y$ and~$\log_{y} x$ are identified with~$r' - r$ and~$r - r'$, respectively, so the LHS of~\eqref{eq: mono tangential} equals~$-(a - a')(r - r')$.
  Hence~\eqref{eq: mono tangential} is equivalent to~$(a - a')(r - r') \geq 0$.

  \proofstep{Case II: $y = r'e_{j}$ with~$r' > 0$ and~$j \neq i$.}
  Let~$y^{\ast} \in Ay$ be identified with~$a' \in a_{j}(r')$.
  The geodesic~$[x,y]$ passes through~$O$, so~$\log_{x} y$ and~$\log_{y} x$ are both identified with~$-(r + r')$, and the LHS of~\eqref{eq: mono tangential} equals~$-(a + a')(r + r')$.
  Hence~\eqref{eq: mono tangential} is equivalent to~$a + a' \geq 0$.

  \proofstep{Case III: $y = O$.}
  Let~$y^{\ast} \in AO$, so that~$y^{\ast} = sd_{k}$ for some~$k$ and~$s \geq 0$.
  Then~$\log_{x} O$ is identified with~$-r$ and~$\log_{O} x = rd_{i}$, so the LHS of~\eqref{eq: mono tangential} equals~$-ar + \langle sd_{k}, rd_{i} \rangle$, which is~$-r(a - s)$ if~$k = i$ and~$-r(a + s)$ if~$k \neq i$.
  Hence~\eqref{eq: mono tangential} is equivalent to~$a \geq s$ if~$k = i$, and to~$a \geq -s$ if~$k \neq i$.
\end{proof}

Let~$A : \Omega \multimap T\Omega$ be a maximally monotone vector field.
For~$i=1,2,3$, define
\begin{align*}
  D_{i} := \set{ r > 0 \mid a_{i}(r) \neq \emptyset }
  \eqsepand
  \alpha_{i} := \inf \bigcup_{r > 0} a_{i}(r)
  \eqfinv
\end{align*}
with the convention~$\inf \emptyset = +\infty$.

\begin{lemma} \label{lem: maxim mono tripod}
  Let~$A : \Omega \multimap T\Omega$ be a maximally monotone vector field.
  \begin{enumerate}[label=\upshape(\alph*), leftmargin=*]
    \item \label{conc a: lem maxim mono tripod}
      If~$\alpha_{i} = -\infty$ for some~$i$, then~$AO = \emptyset$ and~$D_{j} = \emptyset$ for all~$j \neq i$.
    \item \label{conc b: lem maxim mono tripod}
      If~$\alpha_{j} > -\infty$ for all~$j$ and~$D_{i} = \emptyset$, then for every~$s_{0} \geq 0$ there exists~$s \geq s_{0}$ such that~$sd_{i} \in AO$.
    \item \label{conc c: lem maxim mono tripod}
      If~$\alpha_{j} > -\infty$ for all~$j$ and~$D_{i} \neq \emptyset$, then~$\inf D_{i} = 0$.
  \end{enumerate}
\end{lemma}
\begin{proof}
  \ref{conc a: lem maxim mono tripod}
  Let~$\alpha_{i} = -\infty$.
  Suppose that~$AO \neq \emptyset$ so that~$sd_{k} \in AO$ for some~$k$ and~$s \geq 0$.
  By Lemma~\ref{lem: equiv mono tripod}~\ref{cond c: lem equiv mono tripod}, every~$a \in a_{i}(r)$ with~$r > 0$ satisfies either~$a \geq s$ if~$k = i$ or~$a \geq -s$ otherwise.
  In both cases,~$a \geq -s$ and hence the infimum defining~$\alpha_{i}$ gives~$\alpha_{i} \geq -s > -\infty$. This contradicts with the hypothesis that~$\alpha_{i} = -\infty$, hence~$AO = \emptyset$.

  Suppose next that~$j \neq i$ and~$D_{j} \neq \emptyset$.
  So there are~$r' > 0$ and~$a' \in a_{j}(r')$.
  By Lemma~\ref{lem: equiv mono tripod}~\ref{cond b: lem equiv mono tripod},~$a + a' \geq 0$ for every~$a \in a_{i}(r)$ and every~$r > 0$.
  A similar trick shows~$\alpha_{i} \geq -a' > -\infty$, a contradiction.
  Thus we have~$D_{j} = \emptyset$ for all~$j \neq i$.

  \ref{conc b: lem maxim mono tripod}
  Suppose to the contrary, so that~$\sigma := \sup \set{ s \geq 0 \mid sd_{i} \in AO } < +\infty$.
  Put~$a := \max_{j \neq i} \set{0, \sigma, -\alpha_{j}}$ which is finite and nonnegative.
  We claim that the pair~$(e_{i}, a)$ relates monotonically to all elements in~$\gr A$.
  Let~$j \neq i$.
  Since~$\alpha_{j} > -\infty$,~$a_{j}(r') \neq \emptyset$ for some~$r' > 0$.
  We have~$a + a' \geq a + \alpha_{j} \geq 0$.
  It remains to show that~$(e_{i}, a)$ relates monotonically to elements at~$AO$ when it is nonempty.
  Suppose that~$sd_{k} \in AO$, where~$s \geq 0$.
  Since~$\log_{e_{i}}O$ is identified with~$-1$ and~$\log_{O}e_{i} = d_{i}$, the LHS of~\eqref{eq: mono tangential} for~$(e_{i}, a)$ and~$(o, sd_{k})$ is
  \begin{align*}
    -a + \langle sd_{k}, d_{i} \rangle = 
    \begin{cases}
      -a + s  &\text{if~$k = 1$} \eqfinv \\
      -a - s  &\text{otherwise} \eqfinp
    \end{cases}
  \end{align*}
  So the monotonic relation occurs if and only if one has~$a \geq s$ if~$k = 1$, and~$a \geq -s$ otherwise.
  In fact one always has~$a \geq 0 \geq -s$.
  For~$k = i$, we additionally observe that~$s \leq \sigma \leq a$ by definition of~$\sigma$.
  Thus the claim is proved.
  By monotone maximality,~$a \in a_{i}(1)$ and hence~$1 \in D_{i}$.
  This contradicts with~$D_{i} = \emptyset$.

  \ref{conc c: lem maxim mono tripod}
  Assume that~$r_{0} := \inf D_{i} > 0$ and fix~$r \in \ooint{0, r_{0}}$.
  We claim that~$(re_{i}, \alpha_{i})$ relates monotonically to all elements in~$\gr A$.
  For~$a' \in a_{i}(r')$, we have~$r' \geq r_{0} > r$ and~$a' \geq \alpha_{i}$.
  So~$(\alpha_{i} - a_{i})(r - r') \geq 0$.
  For~$a' \in a_{j}(r')$ with~$j \neq i$, we have~$\alpha_{i} + a' \geq \alpha_{i} + \alpha_{j} \geq 0$, where the last inequality follows from Lemma~\ref{lem: maxim mono tripod}~\ref{cond b: lem equiv mono tripod}.
  If~$AO \neq \emptyset$, take any~$sd_{k} \in AO$ where~$s \geq 0$.
  Lemma~\ref{lem: equiv mono tripod}~\ref{cond c: lem equiv mono tripod} gives~$\alpha_{i} \geq s$ for~$k = i$ and~$\alpha_{i} \geq -s$ otherwise.
  By maximality,~$\alpha_{i} \in a_{i}(r)$ and so~$r \in D_{i}$.
  This contradicts~$r < r_{0}$.
\end{proof}

Let~$L_{ij}$ be the geodesic line through rays~$i$ and~$j$ of~$\Omega$, with~$i \neq j$.
It is clearly a flat slice that is isometric to the real line~$\R$.
Suppose that the positive coordinates~$t > 0$ corresponds to the ray~$i$ and the negative coordinates~$t < 0$ to the ray~$j$, with~$t = 0$ identifying the branch point~$O$.

\begin{proposition} \label{prop: max mono extensive}
  Let~$A : \Omega \multimap T\Omega$ and let~$q \in \set{1,2,3}$ be such that~$\alpha_{q}= \min_{i=1,2,3} \alpha_{i}$.
  Then for each~$p \neq q$, the line~$C = L_{pq}$ satisfies~\ref{cond b: thm flat slice principle} of Theorem~\ref{thm: flat slice principle}.
\end{proposition}
\begin{proof}
  Let~$m$ be the third index different from both~$p$ and~$q$.
  Elements of~$\gr A \setminus \gr A_{C}$ are the pairs~$(re_{m}, a')$ with~$r > 0$ and~$a' \in a_{m}(r)$, and the pairs~$(O, sd_{m}) \in \gr A$ with~$s > 0$.
  Let~$(z,w)$ be a pair with~$z \in C$ and~$w \in T_{w}C$ that relates monotonically to all elements of~$\gr A_{C}$.
  We claim that the relation extends to all the pairs in~$\gr A$.
  If~$\alpha_{q} = -\infty$, then~$AO = \emptyset$ and~$D_{m} = \emptyset$ by Lemma~\ref{lem: maxim mono tripod}~\ref{conc a: lem maxim mono tripod} and the claim holds.
  It remains to consider the case with~$\alpha_{j} > -\infty$ for all~$j$.

  \proofstep{Case I: $z \neq O$.}
  Since~$z \in C = L_{pq}$, we have~$z = te_{i}$ with~$i \in \set{p,q}$ and~$t > 0$.
  Identify~$w$ with~$c \in \R$.
  If~$D_{i} = \emptyset$, Lemma~\ref{lem: maxim mono tripod}~\ref{conc b: lem maxim mono tripod} produces~$sd_{i} \in AO$ with~$s > c$.
  This element lies in~$A_{C}O$ and the monotone relation means~$c \geq s$, which is absurd.
  So, it must be the case that~$D_{i} \neq \emptyset$ and Lemma~\ref{lem: maxim mono tripod}~\ref{conc c: lem maxim mono tripod} gives~$t_{1} \in D_{i}$ with~$t_{1} < t$.
  For~$b \in a_{i}(t_{1})$, the monotone relation yields~$(c - b)(t - t_{1}) \geq 0$.
  Hence~$c \geq b \geq \alpha_{i}$ and so~$c + a' \geq \alpha_{i} + \alpha_{m} \geq 0$ for all~$a' \in a_{m}(r)$.
  This shows that~$(z,w)$ is monotonically related to~$(re_{m}, a')$ with~$a' \in a_{m}(r)$.
  Next, if~$sd_{m} \in AO$ where~$s \geq 0$, Lemma~\ref{lem: equiv mono tripod}~\ref{cond c: lem equiv mono tripod} implies that~$-s \leq \alpha_{i} \leq c$.
  This is exactly the monotone relation of~$(z,w)$ to~$(O, sd_{m})$.

  \proofstep{Case II: $z = O$.}
  Again, identify~$w$ with~$c \in \R$.
  We only need to show that~$(z,w)$ is monotonically related to the pairs of the form~$(re_{m}, a')$ where~$r > 0$ and~$a' \in a_{m}(r)$. 
  Note that~$w$ never points towards the ray~$m$, which means~$\langle w, \log_{O} re_{m} \rangle = \langle w, rd_{m} \rangle = -r \abs{c}$ and the LHS of~\eqref{eq: mono tangential} for~$(O,w)$ and~$(re_{m}, a')$ is
  \begin{align*}
    \langle w, rd_{m} \rangle - a'r = -r (a' + \abs{c})
    \eqfinp
  \end{align*}
  Since~$r > 0$, the monotone relation of the above means~$a' \geq - \abs{c}$.

  Assume that~$D_{m} \neq \emptyset$, otherwise there is nothing to prove, so that we have~$\alpha_{q} \leq \alpha_{m} < +\infty$ and this means~$D_{q} \neq \emptyset$.
  Since~$(O,w)$ is related monotinically to elements of~$\gr A_{C}$ on ray~$q$, namely the pairs~$(\delta e_{q}, b)$ with~$\delta > 0$ and~$b \in a_{q}(\delta)$.
  As above,~$\log_{O} \delta e_{q} = \delta d_{q}$ and~$\log_{\delta e_{q}} O = - \delta$ so that the LHS of~\eqref{eq: mono tangential} reduces to~$\langle w, \delta d_{q} \rangle - b\delta = - \delta (c + b)$.
  Monotonicity gives~$b \geq -c$ for any such~$b$, and hence~$\alpha_{q} \geq -c$.
  We then obtain~$a' \geq \alpha_{m} \geq \alpha_{q} \geq -c \geq - \abs{c}$, as desired.
\end{proof}

\begin{theorem} \label{thm: maxim mono surj tripod}
  Every maximally monotone vector field~$A : \Omega \multimap T\Omega$ satisfies the surjectivity condition.
\end{theorem}
\begin{proof}
  Let~$q$ be as in the Proposition~\ref{prop: max mono extensive}.
  Then~$\Omega$ is covered by~$\set{L_{pq}}_{p \neq q}$.
  Each~$L_{pq}$,~$p \neq q$, is isomorphic to a Hilbert space~$\mathcal{H} = \R$.
  Together with Proposition~\ref{prop: max mono extensive}, all the hypotheses of Theorem~\ref{thm: flat slice principle} are satisfied.
  Therefore~$J_{A}^{\lambda}(x) \neq \emptyset$ for all~$x \in L_{pq}$ and~$\lambda > 0$.
  Since every point of~$\Omega$ lies on~$L_{pq}$ for some~$p \neq q$, we conclude that~$J_{A}^{\lambda}(x) \neq \emptyset$ for all~$x \in \Omega$ and~$\lambda > 0$, \ie $A$ satisfies the surjectivity condition.
\end{proof}

\subsection{Hilbert-Hadamard spaces} \label{subsec: HH}

Recall that a~$\CAT(0)$ space is said to be \emph{flat} if the inequality~\eqref{eq: CN ineq} always holds as an equality.
In such a case, we know that~$X$ is isometric to some closed convex set~$K$ of a Hilbert space~$\mathcal{H}$~\cite[Remark~2.1]{SL}.

A complete~$\CAT(0)$ space is said to be a \emph{Hilbert-Hadamard space} if the tangent cone~$T_{x}X$ is flat at every~$x \in X$.
This notion was introduced by~\citet{GWY}, who required instead that each tangent cone~$T_{x}X$ embeds isometrically into a Hilbert space.
As noted in~\cite[Remark~2.1]{SL}, the two definitions are equivalent.
Without loss of generality, we assume that isometry goes through~$\iota_{x}$ which preserves the zero element of~$T_{x}X$.
Since~$\iota_{x}$ preserves the distance to the apex, it preserves~$\norm{u}$ for any~$u \in T_{x}X$.
By the cosine law, it also preserves the inner product.
Let us write~$\mathcal{H}_{x}X$ the \emph{tangent Hilbert space}, which is the Hilbert space spanned by the tangent cone.

For~$x \in X$ and~$g \in \mathcal{H}_{x}X$, define the \emph{open half-space}
\begin{align*}
  U_{x,g} := \set{ y \in X \mid \langle g , \log_{x} y \rangle > 0 }
  \eqfinp
\end{align*}

\begin{definition}
  A Hilbert-Hadamard space~$X$ satisfies the \emph{convex half-space condition} if~$U_{x,g}$ is convex for every~$x \in X$ and~$g \in \mathcal{H}_{x}X$.
\end{definition}

\citet{Horvath} has provided a framework in a general topological space to develop a version of KKM lemma.
The following lemma is a corollary of~\cite[Proposition 1]{Horvath} where the convex hull of~$N$ points plays the role of the~$N$-dimensional singular face of the original statement.
\begin{lemma} \label{lem: KKM}
  Let~$X$ be a complete~$\CAT(0)$ space and~$G : X \multimap X$ have closed values.
  Suppose that~$\conv \set{y_{1},\dots,y_{N}} \subset \bigcup_{i=1}^{N} G(y_{i})$ for any finite subset~$\set{y_{1},\dots,y_{N}} \subset X$.
  Then the family~$\set{ G(y) }_{y \in X}$ has the finite intersection property.
\end{lemma}

Using this KKM lemma, we retrieve the Minty's principle in the setting of a Hilbert-Hadamard space.
The following characterization will be used in the proof of the Minty's principle.
\begin{lemma} \label{lem: characterization resolvent}
  Let~$X$ be a complete~$\CAT(0)$ space and~$A : X \multimap TX$ be a maximally monotone vector field.
  Then~$z \in J_{A}^{\lambda}(x)$ if and only if
  \begin{align} \label{eq: characterization resolvent}
    \lambda^{-1} \langle \log_{z} x , \log_{z} y \rangle + \langle v , \log_{y} z \rangle \leq 0
    \qquad \bigl( \forall (y,v) \in \gr A \bigr)
    \eqfinp
  \end{align}
\end{lemma}
\begin{proof}
  Set~$u := \lambda^{-1} \log_{z} x$.
  Then~\eqref{eq: characterization resolvent} says that~$(z,u)$ is monotonically related to every~$(y,v) \in \gr A$.
  The maximality of~$A$ implies that~$u \in Az$, which means~$z \in J_{A}^{\lambda}(x)$.

  Conversely, suppose that~$z \in J_{A}^{\lambda}(x)$ which means~$u \in Az$.
  The conclusion follows from the monotonicity of~$A$.
\end{proof}

\begin{theorem}[Minty's principle] \label{thm: Minty HH}
  Let~$X$ be a proper Hilbert-Hadamard space satisfying the convex half-space condition.
  Then every maximally monotone vector field~$A : X \multimap TX$ with~$\dom A = X$ satisfies the surjectivity condition.
\end{theorem}
\begin{proof}
  Let~$x \in X$ and~$\lambda > 0$.
  For~$(y,v) \in \gr A$, define
  \begin{align*}
    f_{y,v}(z) := \lambda^{-1} \langle \log_{z} x, \log_{z} y \rangle + \langle v, \log_{y} z \rangle
    \eqsepand
    F(y,v) := \set{z \in X \mid f_{y,v} \leq 0}
    \eqfinp
  \end{align*}
  By Lemma~\ref{lem: characterization resolvent},~$z \in J_{A}^{\lambda}(x)$ if and only if~$z \in F(y,v)$ for all~$(y,v) \in \gr A$.

  We claim that~$\bigcap_{(y,v) \in \gr A} F(y,v) \neq \emptyset$.
  First, define~$G(y) := \bigcap_{v \in Ay} F(y,v)$ for each~$y \in X$.
  Then~$G$ has closed values.
  Now, let~$y_{1}, \dots, y_{N} \in X$ and~$z \in \conv \set{y_{1}, \dots, y_{N}}$.
  Suppose that~$z \not\in G(y_{i})$ for any~$i=1,\dots,N$.
  Then for each~$i$, there exists~$v_{i} \in Ay_{i}$ such that~$f_{y_{i}, v_{i}}(z) > 0$.
  Since~$\dom A = X$,~$Az$ is nonempty and we pick an element~$u \in Az$.
  The monotonicity of~$A$ gives~$\langle v_{i}, \log_{y_{i}} z \rangle \leq - \langle u, \log_{z} y_{i} \rangle$. 
  Putting~$g := \lambda^{-1} \log_{z} x - u \in \mathcal{H}_{z}X$, we obtain
  \begin{align*}
    0 
    < 
    f_{y_{i}, v_{i}}(z) \leq \lambda^{-1} \langle \log_{z} x , \log_{z} y_{i} \rangle
    - \langle u, \log_{z} y_{i} \rangle
    =
    \langle g , \log_{z} y_{i} \rangle
    \eqfinp
  \end{align*}
  This means~$y_{i} \in U_{z, g}$ for each~$i$.
  Since~$U_{z,g}$ is convex, one gets~$z \in \conv\set{y_{1}, \dots, y_{N}} \subset U_{z,g}$.
  This is a contradiction because~$z$ never belongs to~$U_{z,g}$ by definition.
  Therefore,~$z \in \bigcup_{i=1}^{N} G(y_{i})$.
  Applying~Lemma~\ref{lem: KKM},~$\set{G(y)}_{y \in X}$ has the finite intersection property.

  One may further observe that~$f_{y,v}$ is lower semicontinuous, which makes~$F(y,v)$ a closed set.
  If~$X$ is bounded, then~$F(y,v)$ is compact since~$X$ is proper.
  If~$X$ is unbounded, one gets the estimate
  \begin{align*}
    f_{y,v}(z) \geq \frac{1}{2\lambda} \bigl( d(z,x)^{2} + d(z,y)^{2} - d(x,y)^{2} \bigr) - \norm{v} d(y,z)
  \end{align*}
  which shows that~$f_{y,v}(z) \to +\infty$ as~$d(z,y) \to +\infty$.
  Consequently, the set~$F(y,v)$ is contained in a ball of finite radius and hence is compact by the properness of~$X$.
  It follows that~$G(y)$ is compact for every~$y \in X$.

  Since~$\set{G(y)}_{y \in X}$ is a family of compact sets with the finite intersection property, we get
  \begin{align*}
    \bigcap_{(y,v) \in \gr A} F(y,v) = \bigcap_{y \in X} \bigcap_{v \in Ay} F(y,v) = \bigcap_{y \in X} G(y) \neq \emptyset
    \eqfinp
  \end{align*}
  Any element~$z \in \bigcap_{y \in X} G(y)$ satisfies~$z \in J_{A}^{\lambda}(x)$.
  Since~$x \in X$ and~$\lambda > 0$ are arbitrary, the conclusion follows.
\end{proof}

We note here that the convex half-space condition is not true for a~$\CAT(0)$ space in general.
In the following example, we present a Hilbert-Hadamard space that lacks the convex half-space condition but every maximally monotone vector field with full domain satisfies the surjectivity condition.
This also shows that the convex half-space condition is not necessary for the equivalence in the open question.
In the following examples, the readers are invited to consult~\cite{Sakai} for the notions used in the context of Riemannian geometry.

\begin{example} \label{ex: HH not convex half-space}
  Let~$X := \mathbb{H}^{2} \times \R$, which is a Hadamard manifold and hence a proper Hilbert-Hadamard space with~$T_{(p,t)}X = T_{p}\mathbb{H}^{2} \times \R$.
  Fix~$p \in \mathbb{H}^{2}$, a unit vector~$e \in T_{p}\mathbb{H}^{2}$, and~$a > 0$.
  Let~$q := \exp_{p}(ae)$ and let~$\gamma$ be the unit-speed geodesic in~$\mathbb{H}^{2}$ through~$q = \gamma(0)$ orthogonal to~$\ccint{p,q}$.
  For~$s \neq 0$, the points~$p$,~$q$ and~$\gamma(s)$ form a geodesic triangle with a right angle at~$q$.
  Write~$r(s) := d(p, \gamma(s))$ and let~$\theta(s)$ be the angle at~$p$ between~$e$ and~$\log_{p} \gamma(s)$, so that~$\phi(s) := \langle e, \log_{p} \gamma(s) \rangle = r(s) \cos \theta(s)$.
  In this right-angled triangle, the hypotenuse has length~$r(s)$, and the sides adjacent and opposite to the angle~$\theta(s)$ have lengths~$a$ and~$\abs{s}$, respectively.
  The hyperbolic Pythagorean theorem~\cite[Theorem~7.11.1]{Beardon} and the relation between the hypotenuse, an adjacent side and the included angle~\cite[Theorem~7.11.2~(iii)]{Beardon} give
  \begin{align*}
    \cosh r(s) = \cosh a \cosh s
    \eqsepand
    \tanh a = \tanh r(s) \cos \theta(s)
    \eqfinp
  \end{align*}
  Hence
  \begin{align*}
    \phi(s) = r(s) \cos \theta(s) = \tanh a \cdot \frac{r(s)}{\tanh r(s)}
    \eqfinp
  \end{align*}
  Since~$r \mapsto \frac{r}{\tanh r}$ is strictly increasing on~$\ooint{0, \infty}$ and~$r(s) > r(0) = a$ for~$s \neq 0$, we get~$\phi(s) > \phi(0) = a$ for every~$s \neq 0$.

  Now fix~$s > 0$, put~$\delta := \phi(s) - a > 0$ and~$t_{0} := - \frac{1}{2}\bigl(a + \phi(s)\bigr)$, and consider~$x := (p,0)$ and~$g := (e,1) \in \mathcal{H}_{x}X$.
  Since~$\log_{x}(w,t) = (\log_{p} w, t)$, we have~$\langle g, \log_{x}(w,t) \rangle = \langle e, \log_{p} w \rangle + t$.
  The points~$(\gamma(\pm s), t_{0})$ belong to~$U_{x,g}$, because~$\phi(s) + t_{0} = \delta/2 > 0$.
  The midpoint~$(q, t_{0})$ of the geodesic joining them does not belong to~$U_{x,g}$, because~$a + t_{0} = -\delta/2 < 0$.
  Hence~$U_{x,g}$ is not convex.
  On the other hand, every maximally monotone vector field on~$X$ with full domain satisfies the surjectivity condition by~\cite{LLM,LLMW}.
\end{example}

The following is an example of a Hilbert-Hadamard space with the convex half-space condition which is not an Hadamard manifold.
This example provides a case emphasizing the novelty of Theorem~\ref{thm: Minty HH} where the known results do not apply.
\begin{example} \label{ex: HH hyperbolic}
  Let~$X$ be a closed convex subset of the real hyperbolic space~$\mathbb{H}^{n}$ with~$n < \infty$.
  For~$x \in X$, we have~$T_{x}X \subset T_{x}\mathbb{H}^{n} \cong \R^{n}$, and hence~$\mathcal{H}_{x}X \subset T_{x}\mathbb{H}^{n}$.
  Consider the Klein model of~$\mathbb{H}^{n}$ \cite[Chapter~I.6]{BH} in the open unit ball of~$\R^{n}$, with~$x$ placed at the origin.
  In this model, geodesics are straight chords~\cite[pp.~83--84]{BH}, so that convex subsets of~$\mathbb{H}^{n}$ are exactly the Euclidean convex subsets of the ball.
  At the origin, the Riemannian metric coincides with the Euclidean one, and~$\log_{x} y$ is a positive multiple of the Euclidean vector~$y$.
  Hence, for~$g \in \mathcal{H}_{x}X$, we have~$U_{x,g} = X \cap \set{ y \mid \langle g, y \rangle_{\R^{n}} > 0 }$, which is the intersection of two convex sets.
  Thus~$X$ satisfies the convex half-space condition.
  If~$X$ has a nonempty boundary in~$\mathbb{H}^{n}$, \eg a closed half-space, then~$X$ is not a Hadamard manifold, so that the results of~\cite{LLM,LLMW} do not apply directly.
  Moreover,~$X$ is not flat, so Theorem~\ref{thm: flat slice principle} does not apply either.
\end{example}

Finally, we remark that~$\Omega$ is not a Hilbert-Hadamard space, since~$T_{O}\Omega$ is not flat.
Indeed, the unit directions~$d_{1}, d_{2}, d_{3}$ have pairwise distance~$2$ and are all at distance~$1$ from the apex, which is impossible for three distinct vectors in a Hilbert space.
This justifies the complementarity of Theorems~\ref{thm: flat slice principle} and~\ref{thm: Minty HH} mentioned at the beginning of this section.

\footnotesize
\section*{Acknowledgments}

The author acknowledges the financial support provided by the Center of Excellence in Theoretical and Computational Science (TaCS-CoE), Faculty of Science, King Mongkut's University of Technology Thonburi (KMUTT), Bangkok, Thailand, KMUTT.

The author also acknowledges the use of the large language model Claude Opus 5.5 (Anthropic) during the preparation of this paper, for drafting and revising parts of the text and proofs, checking the arguments, and verifying the references.
All content was reviewed and verified by the author, who takes full responsibility for the final work.

\bibliography{cat0_nonlinear_anal.bib}

\end{document}